\documentclass[12pt]{amsart}

\usepackage{amsaddr}
\usepackage{amscd,amssymb}

\title[Arc length of function graphs]{Arc length of function graphs via Taylor's formula}

\newtheorem{theorem}{Theorem}

\newtheorem{lemma}[theorem]{Lemma}
\theoremstyle{definition}
\newtheorem{definition}[theorem]{Definition}
\newtheorem{example}[theorem]{Example}
\newtheorem{remark}[theorem]{Remark}
\newtheorem{problem}[theorem]{Problem}

\begin{document}

\author{Patrik Nystedt}
\address{University West,
Department of Engineering Science, 
SE-461 86 Trollh\"{a}ttan, Sweden}

\email{patrik.nystedt@hv.se}

\begin{abstract}
We use Taylor's formula with Lagrange remainder to prove that functions with bounded second derivative
are rectifiable in the case when polygonal paths are defined by interval sub\-divisions which are equally spaced.
We discuss potential bene\-fits for such an approach in introductory calculus courses.
\end{abstract}

\maketitle




\section{Introduction}

One of the first experiences of measurements that we encounter in our lives
is that of {\it length}. Even young children are involved in many everyday activities
that concern length measurements. Questions such as ''How {\it tall} am I?'' or
''How {\it long} can you jump?'' or
''How {\it far} is it to my friends house?'' arise naturally from them. 
In the early years of schooling we are taught how to measure 
lengths of {\it straight lines} using a ruler and express our findings in appropriate units.
In middle school, we are presented with the problem of measurement of the circumference a the circle
and how to relate this to the length of its diameter.
For many students the transition from understanding straight line
measurements to comprehending length measurement of non-linear curves is not so easily accomplished.
Indeed, it is only natural for them to pose questions such as ''How can we measure something curved
using a straight ruler?'' or ''What do we really mean when we speak of the length of a curve?''.
As teachers, we have to treat these questions seriously, because
when pondering over this, the students are placed in very good company.
Indeed, over the millennia, many of our greatest thinkers failed to provide satisfying answers to such questions.
For instance, the Greek philosopher Aristotle (384-322 BC) stated the following concerning
comparisons of motions along straight lines and along circles:
\begin{quote}
''But, once more, if the motions are comparable, we are met by the difficulty
aforesaid, namely that we shall have a straight line equal to a circle.
But these are not comparable.'' \cite[p. 141]{heath1970}
\end{quote}
With some exceptions (for instance Archimedes rectification of the 
circle using a spiral, see e.g. \cite{richeson2013}), Aristotle's view on these matters
persisted amongst scholars even up to the time of Descartes (1596-1650) who wrote the following in his work
{\it La G\'{e}om\'{e}trie} from 1637:
\begin{quote}
''...the ratios between straight and curved lines are not known,
and I believe cannot be discovered by human minds, and therefore no
conclusion based on such ratios can be accepted as rigorous and exact.''
\cite[p. 91]{smith1954}
\end{quote}
Descartes would only 20 years later be proved wrong on this point by 
Neil who showed how to rectify the semi-cubical parabola $y^3 = a x^2$.
Independently, both van Heuraet and Fermat came to the same conclusion 
within a few years after Neil's discovery \cite{traub1984}.
After that, of course, Newton and Leibniz fully developed the calculus machinery
including formulas for arc length using integrals \cite[p. 217, p. 242]{edwards1979}.

\section{Arc length in calculus teaching}

The first time students are exposed to arc length calculations of general functions is in introductory calculus courses.
In popular calculus books (see e.g. \cite{adams2006,hass2017,stewart2015}) the concept of curve length is typically defined in 
the following way.

\begin{definition}
Let $A$ and $B$ be two points in the plane and let $|AB|$ denote the distance 
between $A$ and $B$.
Let $C$ be a curve in the plane joining $A$ and $B$.
Suppose that we choose points $A = P_0$, $P_1$, $P_2$, $\ldots$, $P_{n-1}$ and $P_n = B$
in order along the curve. The polygonal line $P_0,P_1,P_2,\ldots,P_n$ constructed by joining
adjacent pairs of these points with straight lines forms a polygonal approximation to $C$,
having length $L_n = \sum_{i=1}^n | P_{i-1} P_1 |$.
The curve $C$ is said to be rectifiable if the limit $L$ of $L_n$, 
as $n \to \infty$ and the maximum segment length $|P_{i-1} P_i| \to 0$,
exists. In that case $L$ is called the length of $C$.  
\end{definition}

An obvious pedagogical difficulty for teachers using such a definition 
is that then we are not calculating a limit of a {\it sequence}, 
in the usual sense that the students  are used to, but rather the limit of a {\it net} \cite{olmstead1961}.
Not only is such a definition unsuitable for concrete calculations, for instance 
using computer simulations, but also highly abstract.
Disregarding this difficulty, the typical calculus book (see loc. cit.) will then
state some variant of the following result which is then used
in exercises to calculate lengths of function graphs in particular cases.

\begin{theorem}\label{classicalarclength}
If $f$ is a real-valued function defined on $[a,b]$ with the property that its derivative
exists and is continuous on $[a,b]$, then $f$ is rectifiable on $[a,b]$ and its length 
$L$ equals $\int_a^b \sqrt{1 + f'(x)^2} \ dx$. 
In that case, if $G$ is a primitive function of $\sqrt{1 + (f')^2}$ on $[a,b]$,
then $L = G(b)-G(a)$. 
\end{theorem}

The typical ''proof'' of this result runs as follows.
For the partition $\{ a=x_0 < x_1 < x_2 < \cdots < x_n =b \}$,
let $P_i$ be the point $(x_i , f(x_i))$, $0 \leq i \leq n$.
By the mean-value theorem there exists $c_i \in (x_{i-1},x_i)$ such that
$f(x_i) - f(x_{i-1}) = f'(c_i)(x_i - x_{i-1})$.
A few lines of calculation now yield that 
$L_n = \sum_{i=1}^n \sqrt{1 + f'(c_i)^2} \Delta x_i$
which can be recognised as a Riemann sum for 
$\int_a^b \sqrt{1 + f'(x)^2} \ dx$ which ends the proof by invoking
the fundamental theorem of calculus (FTC).

The problem with this ''proof'' is that it is, in fact, not a proof at all. 
Why? Well, because it relies on the FTC which is not proved in full detail
in any of the popular calculus texts in use today.
Sure, parts of it is proved, but the hardest part concerning the 
convergence of Riemann sums is left out.
The reason for skipping this is that a presentation including all details will be long
and complicated. For instance, in Tao's book \cite{tao2006}
the definition of general Riemann sums and proofs of properties these, including
the FTC, takes more than 30 pages,
excluding an argument for the crucial fact that continuous functions on compact 
intervals are uniformly continuous, which would make the presentation even longer.

We sympathise with the method of ''cheating'' with the theory in calculus courses.
To be honest, we can, of course, not prove every statement made in the course.
However, we feel that leaving out a valid argument concerning such a central fact as the convergence of
Riemann sums should be regarded as cheating at the wrong place.

In a recent article \cite{nystedt2019}, we argue that the integral therefore should
be defined using equally spaced subdivisions of the interval using only left (or right endpoints).
We call the corresponding sums {\it Euler sums}, inspired by the fact that 
Euler \cite[Part I, Section I, Chapter 7]{euler1768} proposed such sums for the approximative calculations of integrals.
In loc. cit., we show, using an idea of Poisson (see \cite{bressoud2011} or \cite{grabiner1983}), 
utilizing Taylor's formula with Lagrange remainder, that the following version of the FTC
easily can be proved in just a few lines of calculation.

\begin{theorem}\label{thmftc}
If $F$ is a real-valued function defined on $[a,b]$ such that its first derivative
exists and is continuous on $[a,b]$, and its second derivative exists and is bounded on $(a,b)$,
then $f=F'$ is integrable on $[a,b]$ and $\int_a^b f(x) dx = F(b)-F(a).$ 
\end{theorem}

\section{Simplified arc length}

In this article, we parallel our investigations in \cite{nystedt2019} and use Euler-like 
sums to define length of function graphs (see Definition \ref{defrectifiable}). 
We prove (see Theorem \ref{thmarclength}), using our version of the FTC, assuming some regularity conditions,
that length of function graphs can be calculated via integrals using the classical formula 
given in Theorem \ref{classicalarclength}.

\begin{definition}\label{defrectifiable}
Suppose that $f$ is a real-valued function defined on an interval $[a,b]$.
For all $n \in \mathbb{N}$ we put $\Delta x = (b-a)/n$, and for all
$k \in \{ 0,1,\ldots,n-1 \}$, we put
$x_k = a + k \Delta x$ and $\Delta y_k = f(x_{k+1})-f(x_k)$.
We say that $L_n = \sum_{k=0}^{n-1} \sqrt{ (\Delta x)^2 +  ( \Delta y_k )^2 }$
is the $n^{\rm th}$ {\it polygonal length} of $f$ on $[a,b]$
and we say that $f$ is {\it rectifiable} on $[a,b]$
if the limit $L = \lim_{n \to \infty} L_n$ exists.
In that case, we call $L$ the {\it arc length} of $f$ on $[a,b]$. 
\end{definition}

The above definition is mathematically crystal clear and
the poly\-gonal lengths of this form are easy for students to calculate in particular cases
(see Section \ref{discussion}).
To prove the main result of the article, we need Taylor's formula with Lagrange remainder,
a result which we now state, for the convenience of the reader.

\begin{theorem}\label{thmtaylor}
Let $n$ be a non-negative integer. 
If $f$ is a real-valued function defined on $[a,b]$ such that its $n^{th}$ derivative exists,
is continuous on $[a,b]$, and is differentiable on $(a,b)$,
then there exists $c \in (a,b)$ such that
$$f(b) = \sum_{j=0}^n \frac{ f^{(j)}(a) }{ j! } (b-a)^i + \frac{ f^{(n+1)}(c) }{ (n+1)! } (b-a)^{n+1}.$$
\end{theorem}

\begin{proof}
For a short proof, see e.g. \cite{hardy1908,nystedt2019,olmstead1961}.
\end{proof}

In the proof of our main result, we also need the following lemma.

\begin{lemma}\label{lemmaABCD}
If $A$, $B$ and $C$ are real numbers, with $A>0$,
then there is a real number $D$, between $0$ and $C$, such that 
$$\sqrt{ A + (B + C)^2 } = \sqrt{ A + B^2 } + \frac{ (B+D)C }{ \sqrt{A + (B+D)^2} }.$$
\end{lemma}

\begin{proof}
Define the function $g : \mathbb{R} \to \mathbb{R}$ by $g(x) = \sqrt{ A + ( B + x )^2 }$, for $x \in \mathbb{R}$.
Since $A > 0$, the function $g$ is differentiable at all $x \in \mathbb{R}$ with derivative
$g'(x) = \frac{ (B + x) }{ \sqrt{A + (B + x)^2} }$.
The claim now follows from Theorem \ref{thmtaylor} with $n=0$, $a = 0$ and $b = C$
(that is, the mean value theorem). 
\end{proof}

\begin{theorem}\label{thmarclength}
If $f$ is a real-valued function defined on $[a,b]$ such that its first derivative
exists and is continuous on $[a,b]$, its second derivative exists and 
is bounded on $(a,b)$, then $f$ is rectifiable on $[a,b]$ if and only
if the function $\sqrt{ 1 + (f')^2 }$ is integrable on $[a,b]$.
In that case, the length $L$ of $f$ on $[a,b]$ equals
$\int_a^b \sqrt{ 1 + f'(x)^2 } \ dx$.
If, in addition, $\sqrt{1 + (f')^2}$ has an antiderivative $G$ on $[a,b]$,
then $L = G(b)-G(a)$. 
\end{theorem}

\begin{proof}
We use the notation introduced earlier.
From Theorem \ref{thmtaylor} with $n=1$, we get that
$$\Delta y_k / \Delta x = f'(x_k) + f''(c) \Delta x/2$$
for some $c \in (x_k,x_{k+1})$, depending on $k$ and $\Delta x$, for $k \in \{ 0,\ldots,n-1 \}$.
Thus, from Lemma \ref{lemmaABCD}, it follows that
\begin{eqnarray*}
\sqrt{ 1 + (\Delta y_k / \Delta x)^2 } &=& 
\sqrt{ 1 + ( f'(x_k) + f''(c) \Delta x / 2 )^2 } \\
&=& \sqrt{ 1 + f'(x_k)^2 } + \frac{(f'(x_k)^2 + D) f''(c) \Delta x /2}{ \sqrt{ 1 + ( f'(x_k)+D )^2 } } \\
\end{eqnarray*}
for some real number $D$ between $0$ and $f''(c) \Delta x / 2$. Hence
$$
L_n = \sum_{k=0}^{n-1} \sqrt{ (\Delta x)^2 + (\Delta y_k)^2 } = \sum_{k=0}^{n-1} \sqrt{ 1 + (\Delta y_k / \Delta x)^2 } \Delta x $$
$$ = \sum_{k=0}^{n-1} \sqrt{ 1 + f'(x_k)^2 } \Delta x + 
\sum_{k=0}^{n-1} \frac{(f'(x_k)+D) f''(c) (\Delta x)^2 /2}{ \sqrt{ 1 + ( f'(x_k)+D )^2 } } $$
which proves the claim, since
$$\left| \sum_{k=0}^{n-1} \frac{(f'(x_k) + D) f''(c) (\Delta x)^2 /2}{ \sqrt{ 1 + ( f'(x_k)+D )^2 } } \right|
\leq \frac{(\Delta x)^2}{2} \sum_{k=0}^{n-1} | f''(c) | \leq \frac{M (b-a)^2}{2n} \to 0,$$
as $n \to \infty$,
for any $M$ satisfying $|f''(x)| \leq M$ when $a < x < b$.
The last part follows from Theorem \ref{thmftc}.
\end{proof}

\begin{remark}
From the above proof, we immediately get the error bound 
$$|L - L_n| \leq \frac{M(b-a)^2}{2n},$$
for all $n \in \mathbb{N}$,
where $M = {\rm sup} \{ \ |f''(x)| \ ; \ a < x < b \ \}$,
for the $n^{\rm th}$ polygonal length.
\end{remark}

\section{Primitives of $\sqrt{1+(f')^2}$}\label{antiderivatives}

It seems to be a common opinion among mathematics teachers that there are few examples of 
functions $f$ for which $\sqrt{1 + (f')^2}$ has a primi\-tive function.
In this section, we show that this is far from true by recalling two large classes of such functions.

\subsection{The examples of Neil, van Heuraet and Fermat}

All of the persons mentioned above considered rectification of curves of the type
$f(x)^n = a x^{n+1}$, for positive integers $n$ and positive real numbers $a$. 
Here, we will not follow their original approaches, but instead use modern tools
from a typical calculus class to investigate this problem. 
First of all, by taking $n^{\rm th}$ roots
we can always rewrite the equation as $f(x) = b x^{1+1/n}$ for a positive real number $b$
(we assume that $x$ and $y$ are non-negative).
Therefore, $\sqrt{1 + f'(x)^2} = \sqrt{1 + c x^{2/n}}$
for some positive real number $c$.
Next, we make the substitution $s = \sqrt{c} x^{1/n}$ so that
$$\sqrt{1 + c x^{2/n}} \ dx = e s^{n-1} \sqrt{1 + s^2} \ ds$$
for some positive real number $e$.
It is well known that it is always possible to find a primitive 
function to an expression which is rational in $s$ and $\sqrt{1+s^2}$
by making the substitution $t = s + \sqrt{1+s^2}$.
Indeed, from the equality $(t-s)^2 = 1 + s^2$ we get that $s = (t^2-1)/2t$ and thus 
$$\sqrt{1+s^2} = t - s = t - (t^2-1)/2t = (t^2+1)/2t.$$
From the equality $s = (t^2-1)/2t$ we get that 
$$ds/dt = (2t \cdot 2t - (t^2-1)2)/(2t)^2 = (t^2+1)/2t^2.$$
Therefore
$$\int s^{n-1} \sqrt{1 + s^2} \ ds = 
\int \frac{(t^2-1)^{n-1}}{(2t)^{n-1}} \cdot
\frac{ t^2+1 }{ 2t } \cdot \frac{t^2+1}{ 2t^2} \ dt$$
$$= 2^{-n-1} \int (t^2-1)^{n-1}(t^4 + 2t^2 + 1) t^{-n-2} \ dt.$$
If we expand the product in the last integral we can write the 
integrand as a sum of powers of $t$ which, of course, is easily integrated.
To illustrate the above procedure, we will carry out this analysis,
in complete detail, in a few cases.

\subsubsection*{The case when $n=1$ and $a=1/2$}
This is the problem of the rectification of the parabola $f(x) = x^2/2$.
In this case $c=1$ and $x=s$ and 
the integral that we seek therefore equals
$$\int \sqrt{1 + x^2} \ dx =  2^{-2} \int (t^4 + 2t^2 + 1) t^{-3} \ dt$$
$$= \frac{1}{4} \int t + 2t^{-1} + t^{-3} \ dt = 
t^2/8 + {\rm log}(t)/2 - t^{-2}/8 + C.$$
To simplify this result, we note that
$$t^2 = 2x^2 + 1 + 2x \sqrt{1+x^2}$$
and
$$(x + \sqrt{x^2+1})(x - \sqrt{x^2+1}) = -1$$
so that
$$t^{-1} = \sqrt{x^2+1} - x$$
which in turn implies that 
$$t^{-2} = 2x^2 + 1 - 2x \sqrt{1+x^2}.$$
All of this finally implies that
$$\int \sqrt{1 + x^2} \ dx = x\sqrt{1+x^2}/2 + {\rm log}( x + \sqrt{1+x^2} )/2 + C.$$

\subsubsection*{The case when $n=2$ and $a=2/3$}
This is the problem of the rectification of the semicubical parabola $f(x)^2 = 4x^3/9$.
In this case we get $f(x) = 2x^{3/2}/3$ so that 
$\sqrt{1 + f'(x)^2} = \sqrt{1 + x}$. Here we could, in theory,
follow the general procedure suggested previously.
However, that would lead to an unnecessarily long calculation
since we immediately see that the sought after integral equals
$$\int \sqrt{1+x} \ dx = 2(1 + x)^{3/2}/3 + C.$$

\subsubsection*{The case when $n=3$ and $a = 3/4$}
This is the problem of the rectification of the curve $f(x) = 3x^{4/3}/4$.
In this case $c=1$ and $x^{1/3} = s$ so that $e=3$ and 
the integral that we seek therefore equals
$$\int \sqrt{1 + x^{2/3}} \ dx = 3 \cdot 2^{-4} \int (t^2-1)^2 (t^4 + 2t^2 + 1) t^{-5} \ dt$$
$$= \frac{3}{16} \int t^3 - 2 t^{-1} + t^{-5} \ dt = 3 t^4/64 - 3 {\rm log}(t)/8 - 3 t^{-4}/64 + C.$$
From the first example, we get that
$$t^4 = 8s^4 + 8s^2 + 1 + 4s(2s^2+1)\sqrt{1+s^2}$$
and
$$t^{-4} = 8s^4 + 8s^2 + 1 - 4s(2s^2+1)\sqrt{1+s^2}$$
so that 
$$\int \sqrt{1 + x^{2/3}} \ dx = 3 s(2s^2+1)\sqrt{1+s^2}/8 - 3 {\rm log}( s + \sqrt{1+s^2} )/8 + C$$
$$= 3 x^{1/3} (2 x^{2/3}+1)\sqrt{1+x^{2/3}}/8 - 3 {\rm log}( x^{1/3} + \sqrt{1 + x^{2/3}} )/8 + C  $$

\subsection{Pythagorean triples}

Suppose that we seek two functions $p$ and $q$ such that $f' = p/q$
and $1 + (f')^2 = g^2$ where $g$ is some function to which we can find a primitive function $G$. 
This implies that $1 + p^2/q^2 = g^2$ or equivalently that $(p^2 + q^2)/q^2 = g^2$.
One way to accomplish this is if $p^2 + q^2 = r^2$ for some function $r$ of reasonably simple type. 
This means that $(p,q,r)$ is a Pythagorean triple of functions.
It is a classical result in number theory that such triples, consisting of integers,
can be parametrized by $p = k(m^2-n^2)$, $q = k(2mn)$ and $r = k(m^2+n^2)$, where
$k$, $m$ and $n$ are positive integers with $m > n$, and with $m$ and $n$ 
coprime and not both odd (see e.g. \cite{long1972}).
In \cite{kubota1972} Kubota 
has shown that the same kind of result holds in any unique factorization 
domain (UFD). In particular, it holds for polynomial rings $\mathbb{R}[X]$,
since they are Euclidean domains and hence UFD's.
The bottom line is that we can use this kind of parametrization to yield examples
of rectifiable curves in the following way.
Choose any functions $m$ and $n$ and put $p = m^2-n^2$ and $q = 2mn$.
Take a function $f$ such that $f' = p/q = m/2n - n/2m$.
Then $\sqrt{1 + (f')^2} = 
\sqrt{1 + (m/2n - n/2m)^2} =
\sqrt{1 + (m/2n)^2 - 1/2 + (n/2m)^2 } =
\sqrt{ (m/2n + n/2m)^2 } = m/2n + n/2m$
so that 
$$G(x) = \int \sqrt{1 + f'(x)^2} \ dx = \int m/2n + n/2m \ dx.$$
Let us illustrate the above algorithm in three examples.

\begin{example}
A problem which often comes up in calculus textbooks is to calculate 
the length of a portion of the hyperbolic cosine function. Based on our calculations above,
it is easy too see why. Indeed, if we put $f(x) = \cosh(x)$, then
$f'(x) = \sinh(x) = m/(2n) - n/(2m)$
if we put $m=e^x$ and $n=1$. Therefore, we get that
$$G(x) = \int m/(2n) + n/(2m) \ dx = \int \cosh(x) \ dx = \sinh(x) + C.$$
The corresponding task for the students could therefore be:
\begin{problem}
Show that the length of $$f(x)=\cosh(x)$$ over the interval $[0,1]$ equals
$$e/2 - 1/(2e).$$
\end{problem}
\end{example}

\begin{example}
Take $m = 4x$ and $n = x^2+1$. Then we need to find $f$ 
so that $f'(x) = m/(2n) - n/(2m) = 4x/(2x^2+2) - x/8 - 1/(8x)$.
We choose $f(x) = {\rm log}( 2x^2 + 2 ) - x^2/16 - {\rm log}(x)/8$.
Then, from the above, we get that  
$$G(x) = \int \sqrt{1 + f'(x)^2} \ dx = \int m/2n + n/2m \ dx$$
$$= \int 4x/(2x^2+2)  + x/8 + 1/(8x) \ dx = 
{\rm log}(2x^2+2) + x^2/16 + {\rm log}(x)/8 + C.$$
Now we can construct a challenging task for the students:
\begin{problem}
Show that the length of $$f(x)= {\rm log}( 2x^2 + 2 ) - x^2/16 - {\rm log}(x)/8$$ 
over the interval $[1,2]$ equals $$3/16 + {\rm log}(5) - 7 {\rm log}(2)/8.$$
\end{problem} 
\end{example}

\begin{example}
Take $m = (x+2)^2$ and $n = (x+1)(x^2+1)$. 
Then we need to find $f$ so that 
$$f'(x) =  m/(2n) - n/(2m) = \frac{ (x+2)^2 }{ 2(x+1)(x^2+1) } - \frac{ (x+1)(x^2+1) }{ 2(x+2)^2 }.$$ 
Since
$$\frac{ (x+2)^2 }{ 2(x+1)(x^2+1) } = \frac{ x }{ 4(x^2+1) } + \frac{ 7 }{ 4(x^2+1) }$$ 
and
$$\frac{(x+1)(x^2+1)}{2(x+2)^2} = x/2 - 3/2 - \frac{5}{2(x+2)^2} + \frac{9}{2(x+2)}$$
we can choose
$$f(x) = \frac{{\rm log}(x^2+1)}{8} + \frac{7{\rm tan}^{-1}(x)}{4} - \frac{x^2}{4}
+ \frac{3x}{2} - \frac{5}{2(x+2)} - \frac{9 {\rm log}(x+2)}{2} + C.$$
Now we can construct a really challenging task for the students:
\begin{problem}
Show that the length of 
$$f(x) = \frac{{\rm log}(x^2+1)}{8} + \frac{7{\rm tan}^{-1}(x)}{4} - \frac{x^2}{4}
+ \frac{3x}{2} - \frac{5}{2(x+2)} - \frac{9 {\rm log}(x+2)}{2}$$
over the interval $[0,1]$ equals
$$\frac{7 \pi}{16} - \frac{5}{3} + \frac{9{\rm log}(3)}{2} - \frac{33{\rm log}(2)}{8}.$$
\end{problem}
\end{example}

\section{Discussion}\label{discussion}

In this article, we have presented a simplified definition of arc length 
as a limit of polygonal sums where the subdivision of the interval is uniform.
We feel that such an approach would support the students' learning of calculus
for many reasons. 

First of all, we have provided a complete proof
that the polygonal lengths converge precisely when the associated integral 
$$\int_a^b \sqrt{1 + f'(x)^2} \ dx$$ exists. In many popular calculus books
the proof of this fact is incomplete since convergence of the nets associated 
to general Riemann sums is not proved.

Secondly and perhaps more importantly, the students can, using a simple computer program,
easily calculate approximations of our simplified  polygonal lengths,
before using the formula $$L = \int_a^b \sqrt{1 + f'(x)^2} \ dx.$$
For instance, suppose the students are given the task
of calculating the arc length of  $f(x) = 2 x^{3/2} / 3$ over the interval $[3,8]$. 
For $n \in \mathbb{N}$ we have that $\Delta x = 5/n$ and thus 
$$L_n = \sum_{k=0}^{n-1} \sqrt{ \frac{25}{n^2} + \left( \frac{2}{3} \left( 3 + \frac{5k+5}{n} \right)^{3/2} - 
\frac{2}{3} \left( 3 + \frac{5k}{n} \right)^{3/2} \right)^2 }.$$
Using a computer program, rounding off to four decimal places, we get 
$$L_1 \approx 12.6508 \quad L_2 \approx 12.6622 \quad L_3 \approx 12.6646 \quad L_4 \approx 12.6655$$
$$L_5 \approx 12.6659 \quad L_{10} \approx 12.6665 \quad L_{20} \approx 12.6666 \quad L_{100} \approx 12.6666$$
which strongly suggests that $L = 38/3$.
After this the students can try to make the exact calculation,
which, as we saw before, is the rectification of the semicubical parabola.
Namely, since $f'(x)^2 = x$, we get, 
using theorem \ref{classicalarclength}, that
$$L = \int_3^8 \sqrt{ 1 + x } \ dx = \left[ \frac{2(1+x)^{3/2}}{3} \right]_3^8 = 
\frac{2 \cdot 9^{3/2}}{3} - \frac{2 \cdot 4^{3/2}}{3} = \frac{38}{3}$$
which confirms what the students guessed.
The students could then move on to try to calculate the length
of the parabola $f(x) = x^2/2$ over the interval $[0,1]$.
Again, making approximative calculations, we have $\Delta x = 1/n$ and thus
$$L_n = \sum_{k=0}^{n-1} \sqrt{ \frac{1}{n^2} + \frac{1}{4} \left( \left( \frac{k+1}{n} \right)^2 - 
\left( \frac{k}{n} \right)^2 \right)^2 }.$$
Using a computer program, rounding off to four decimal places, we get 
$$L_1 \approx 1.1180 \quad L_2 \approx 1.1404 \quad L_3 \approx 1.1445 \quad 
L_4 \approx 1.1459 \quad L_5 \approx 1.1466$$
$$L_{10} \approx 1.1475 \quad L_{20} \approx 1.1477 \quad L_{100} \approx 1.1478 \quad
L_{200} \approx 1.1478.$$
After this, the students could try to calculate the exact value of the integral.
From the discussion in the previous section this is the length of the parabola which equals
$$\int_0^1 \sqrt{1 + x^2} \ dx = \sqrt{2}/2 + {\rm log}(1 + \sqrt{2})/2.$$
Finally, the students could try to calculate the length of $f(x) = x^3/3$ over the interval $[0,1]$.
Numerically, they would easily get $L_{100} = 1.0894$, rounding off to four decimal places.
However, when considering the exact length calculation, they have to deal with the integral
$$\int_0^1 \sqrt{1 + x^4} \ dx$$ which involves elliptic integrals (see e.g. \cite{hancock1958})
and is impossible to calculate exactly using the elementary functions.
It is our firm belief that students should be subjected to the calculation
of such integrals in a typical calculus course, in order for them to appreciate
the numerical calculations, which, after all, are crucially important for them
in a future work-life as e.g. engineers.


\vfill\eject

\end{document}